\documentclass[11pt]{amsart}
\usepackage{amsmath}
\usepackage{amssymb}
\usepackage{amsthm}
\usepackage{graphicx}
\usepackage{caption}
\usepackage{geometry}
\usepackage{color}
\usepackage{bbm}
\usepackage{paralist}
\usepackage{setspace}
\usepackage{tikz}
\usetikzlibrary{arrows,shapes,trees}

\makeatletter
\@addtoreset{equation}{chapter}
\makeatother
\renewcommand\thesection{\arabic{section}}

\renewcommand\theequation{\thesection.\arabic{equation}}

\newenvironment{bemerkung}{\refstepcounter{equation}\textbf{Remark \theequation.}}{}

\newenvironment{definition/proposition}{\refstepcounter{equation}\textbf{Definition/Proposition \theequation.}}{}
\newenvironment{definition/proposition 2}{\refstepcounter{equation}\textbf{Definition/Proposition \theequation} (vgl. z.B. \cite{pw})\textbf{.}}{}
\newtheorem{satz}[equation]{Theorem}
\newtheorem{lemma}[equation]{Lemma}
\newtheorem{korollar}[equation]{Corollary}
\newtheorem{proposition}[equation]{Proposition}

\newcommand{\RR}{\mathbb{R}}

\newcommand{\NN}{\mathbb{N}}
\newcommand{\ZZ}{\mathbb{Z}}

\newcommand{\IInc}{\mathrm{Inc}}

\newcommand{\iid}{\mathrm{id}}

\newcommand{\iin}{\mathrm{in}}

\newcommand{\IInd}{\mathrm{Ind}}

\newcommand{\lcm}{\mathrm{lcm}}
\newcommand{\lt}{\mathrm{lt}}

\newcommand{\rank}{\mathrm{rank}}
\newcommand{\KK}{\mathbb{K}}
\newcommand{\DegMin}{\prec^{\mathrm{d}}_{\mathrm{min}}}
\newcommand{\RevdegMin}{\prec^{\mathrm{rd}}_{\mathrm{min}}}
\newcommand{\Min}{\prec_{\mathrm{min}}}
\newcommand{\Maxi}{\prec_{\mathrm{max,i}}}
\newcommand{\DegMaxi}{\prec^{\mathrm{d}}_{\mathrm{max,i}}}
\newcommand{\RevdegMaxi}{\prec^{\mathrm{rd}}_{\mathrm{max,i}}}
\newcommand{\Deg}{\prec^{\mathrm{d}}}
\newcommand{\MaxEins}{\prec_{\mathrm{max,1}}}
\newcommand{\MaxZwei}{\prec_{\mathrm{max,2}}}
\newcommand{\DegMaxZwei}{\prec^{\mathrm{d}}_{\mathrm{max,2}}}
\newcommand{\RevdegMaxZwei}{\prec^{\mathrm{rd}}_{\mathrm{max,2}}}

\begin{document}
\pagenumbering{arabic}

\title{Finite numbers of initial ideals in non-Noetherian polynomial rings}
\author{Felicitas Lindner}
\address{Universit\"at Marburg, Fachbereich Mathematik und Informatik,
35032 Marburg, Germany}
\email{lindner5@mathematik.uni-marburg.de}
\maketitle

\begin{abstract}In this article, we generalize the well-known result that ideals of Noetherian polynomial rings have only finitely many initial ideals to the situation of ascending ideal chains in non-Noetherian polynomial rings. More precisely, we study ideal chains in the polynomial ring $R=\KK[x_{i,j}\,|\,1\leq i\leq c,j\in\NN]$ that are invariant under the action of the monoid $\IInc(\NN)$ of strictly increasing functions on $\NN$, which acts on $R$ by shifting the second variable index. We show that for every such ideal chain, the number of initial ideal chains with respect to term orders on $R$ that are compatible with the action of $\IInc(\NN)$ is finite. As a consequence of this, we will see that $\IInc(\NN)$-invariant ideals of $R$ have only finitely many initial ideals with respect to $\IInc(\NN)$-compatible term orders. The article also addresses the question of how many such term orders exist. We give a complete list of the $\IInc(\NN)$-compatible term orders for the case $c=1$ and show that there are infinitely many for $c >1$. This answers a question by Hillar, Kroner, Leykin.
\end{abstract}

\section{Introduction}
It has long been known that for ideals of polynomial rings in finitely many variables, the number of initial ideals with respect to arbitrary term orders is finite (e.g. \cite{rob}, Lemma 2.6). As this result relies on the Noetherianity of such polynomial rings, it cannot be transferred to ideals of polynomial rings in infinitely many variables in general. However, more recent results show that for certain non-Noetherian polynomial rings, there are classes of ideals satisfying a weaker kind of Noetherianity, namely Noetherianity up to the action of certain monoids. Thus, it seems worthwhile to try and generalize the result on finiteness of numbers of initial ideals in the Noetherian case to this class of ideals.

Let $R:=\KK[x_{i,j}\,|\,i\in[c],j\in\NN]$ be the polynomial ring over an arbitrary field $\KK$ in the variables indexed by $[c]\times\NN$, where $\NN:=\lbrace 1,2,3,...\rbrace$ denotes the set of natural numbers, $c\in\NN$ is any fixed number and $[c]:=\lbrace 1,...,c\rbrace$. On $R$, we can define an action of the monoid
\begin{equation*}
\IInc(\NN):=\left\lbrace p:\NN\rightarrow\NN\,|\,p(n)<p(n+1)\text{ for all }n\in\NN\right\rbrace
\end{equation*}
of strictly increasing functions on $\NN$ by $\KK$-linear extension of the map
\begin{equation*}
x_{i_1,j_1}^{e_1}\cdot...\cdot x_{i_r,j_r}^{e_r}\mapsto p\cdot(x_{i_1,j_1}^{e_1}\cdot...\cdot x_{i_r,j_r}^{e_r}):=x_{i_1,p(j_1)}^{e_1}\cdot...\cdot x_{i_r,p(j_r)}^{e_r}
\end{equation*}
for every $p\in\IInc(\NN)$. Let $R_n:=\KK[x_{i,j}\,|\,i\in[c],j\in[n]]$ and
\begin{equation*}
\IInc(\NN)_{m,n}:=\left\lbrace p\in\IInc(\NN)\,|\,p(m)\leq n\right\rbrace
\end{equation*}
for each pair of natural numbers $m\leq n$. We call a sequence of ideals $J_{\circ}=J_1\subseteq J_2\subseteq...$, where each $J_n$ is an ideal of $R_n$, an $\IInc(\NN)$-invariant ideal chain in $R$ if for every $m\leq n$, we have
\begin{equation*}
\IInc(\NN)_{m,n}\cdot J_m\subseteq J_n.
\end{equation*}
In \cite{hs} it was shown that every $\IInc(\NN)$-invariant ideal chain $J_\circ$ in $R$ stabilizes up to the action of $\IInc(\NN)$, i.e. there is an index $N\in\NN$ satisfying
\begin{equation}
\left\langle\IInc(\NN)_{N,n}\cdot J_N\right\rangle_{R_n}=J_n
\label{intro: stability index}
\end{equation}
for every $n\geq N$. We call the minimal $N$ satisfying (\ref{intro: stability index}) the stability index of $J_\circ$ and denote it by $\IInd(J_\circ)$.

Let $\preceq$ be a term order on $R$, i.e. a total order on the monomials of $R$ respecting multiplication and satisfying $1\preceq f$ for every monomial $f$. If $\preceq$ has the additional property that
\begin{equation*}
f\preceq g\Rightarrow p\cdot f\preceq p\cdot g
\end{equation*}
for all monomials $f,g\in R$ and every $p\in\IInc(\NN)$, then we call $\preceq$ an $\IInc(\NN)$-compatible term order on $R$. If $\preceq$ is $\IInc(\NN)$-compatible, then for every polynomial $f\in R$, the leading monomial $\iin_\preceq(f)$ of $f$ with respect to $\preceq$ satisfies
\begin{equation*}
\iin_{\preceq}(p\cdot f)=p\cdot\iin_{\preceq}(f).
\end{equation*}
This implies that for every $\IInc(\NN)$-invariant ideal chain $J_{\circ}$ in $R$, the chain of initial ideals
\begin{equation*}
\iin_{\preceq}(J_\circ):=\iin_{\preceq}(J_1)\subseteq\iin_{\preceq}(J_2)\subseteq...
\end{equation*}
is $\IInc(\NN)$-invariant, too, and therefore stabilizes. Thus, we can define the set
\begin{equation*}
\mathrm{I}(J_\circ):=\left\lbrace\IInd(\iin_{\preceq}(J_\circ))\,|\,\preceq\text{ is an }\IInc(\NN)\text{-compatible term order on }R\right\rbrace
\end{equation*}
of stability indices of initial ideal chains of $J_\circ$ with respect to $\IInc(\NN)$-compatible term orders. In this article, we will prove the following statement:

\begin{satz}
For every $\IInc(\NN)$-invariant ideal chain $J_\circ$ in $R$, the set $\mathrm{I}(J_\circ)$ is bounded above (and, thus, finite).
\label{theorem: boundedness of stability indices}
\end{satz}

Note that as the global stability index $\IInd(J_\circ)$ of the ideal chain $J_\circ$ can be smaller than $\max(\mathrm{I}(J_\circ))$ (see Remark \ref{example: global index smaller than maximum}), the perhaps obvious idea to prove Theorem \ref{theorem: boundedness of stability indices} by showing that $\mathrm{I}(J_\circ)$ is bounded by $\IInd(J_\circ)$ must fail. Therefore, we have to use a different approach.

Theorem \ref{theorem: boundedness of stability indices} has two interesting consequences in terms of statements on finiteness of numbers of initial ideals or, respectively, initial ideal chains: In Theorem \ref{weitere aspekte: charakterisierungen}, we will see that the number of initial ideal chains of $\IInc(\NN)$-invariant ideal chains in $R$ with respect to $\IInc(\NN)$-compatible term orders is finite. As a consequence of this, the number of initial ideals of $\IInc(\NN)$-invariant ideals of $R$ with respect to $\IInc(\NN)$-compatible term orders is finite, too, see Corollary \ref{corollary: finitely many initial ideals}.

Of course, Theorem \ref{theorem: boundedness of stability indices} would be insubstantial if there were only finitely many $\IInc(\NN)$-compatible term orders on $R$. For $c\geq 2$, we can easily construct an infinite number of $\IInc(\NN)$-compatible term orders: Choose any term order $\preceq'$ on the polynomial ring $\KK[x_{1,1},...,x_{c,1}]$. For every monomial $f\in R$, there is a decomposition $f=x_1^{a(1)}\cdot...\cdot x_n^{a(n)}$, where $a(i)=(a(i)_1,...,a(i)_c)\in\NN_0^c$ and $x_i^{a(i)}:=x_{1,i}^{a(i)_1}\cdot...\cdot x_{c,i}^{a(i)_c}$. Set
\begin{equation*}
x_1^{a(1)}\cdot...\cdot x_n^{a(n)}\prec x_1^{b(1)}\cdot...\cdot x_n^{b(n)}:\Leftrightarrow x_1^{a(i)}\prec' x_1^{b(i)}\text{ for }i=\min\lbrace j\,|\,a(j)\neq b(j)\rbrace.
\end{equation*}
This obviously defines an $\IInc(\NN)$-compatible term order on $R$, and if we choose two distinct term orders $\preceq'_1,\preceq'_2$ of $\KK[x_{1,1},...,x_{c,1}]$, then the corresponding term orders $\preceq_1,\preceq_2$ on $R$ are distinct, too. As there are uncountably many distinct term orders on $\KK[x_{1,1},...,x_{c,1}]$, our claim follows. 

For $c=1$, in contrast, there are only finitely many $\IInc(\NN)$-compatible term orders on $R$. Note that the above example of an $\IInc(\NN)$-compatible term order can be applied to the case $c=1$, too, using for $\preceq'$ the only term order there is on the polynomial ring $\KK[x]$ in one variable. This yields the term order
\begin{equation*}
x_1^{a_1}\cdot...\cdot x_n^{a_n}\prec x_1^{b_1}\cdot...\cdot x_n^{b_n}:\Leftrightarrow a_i<b_i\text{ for }i=\min\lbrace j\,|\,a_j\neq b_j\rbrace,
\end{equation*}
i.e. a term order of lexicographic type. We will see in Theorem \ref{theorem: preorders on R} that, essentially, every $\IInc(\NN)$-compatible term order on $R$ for $c=1$ is of this type, resulting in a number of only six distinct $\IInc(\NN)$-compatible term orders. This answers Question 5.5 \cite{hkl} by Hillar, Kroner, Leykin.

The article is organized as follows: We begin with some technical preparations in Section \ref{section: technicalities} needed for the proof of Theorem \ref{theorem: boundedness of stability indices}, which we will give in Section \ref{section 1}. Section \ref{section 1} also contains the proofs for the finiteness results mentioned above. In Section \ref{section 2} we will then study the question of what the $\IInc(\NN)$-compatible term orders are in the case $c=1$. Here, we will not only consider term orders but the larger class of monomial preorders, where we use the concept of a monomial preorder introduced in \cite{kemp}. In Theorem \ref{theorem: preorders on R}, we will give a full classification of $\IInc(\NN)$-compatible monomial preorders on $R$ for $c=1$, comprising a complete list of the $\IInc(\NN)$-compatible term orders.

%----------------------------------------------------------------------------------------------------------------%
%----------------------------------------------------------------------------------------------------------------%
\section{Technical preparations}\label{section: technicalities}
Here and in the section that follows, the number $c$ from the definitions of $R$ and $R_n$ is an arbitrary natural number. We start this section with some observations concerning the monoid $\IInc(\NN)$ and its action on $R$.

\begin{lemma}[cf. \cite{nr}, Proposition 4.6]
Let $l\leq m\leq n$ be natural numbers. Then
\begin{equation*}
\IInc(\NN)_{l,n}=\IInc(\NN)_{m,n}\circ\IInc(\NN)_{l,m},
\end{equation*}
meaning that for every $p_1\in\IInc(\NN)_{l,m}$, $p_2\in\IInc(\NN)_{m,n}$ we have $p_2\circ p_1\in\IInc(\NN)_{l,n}$, and every element $p\in\IInc(\NN)_{l,n}$ has such a decomposition.
\label{dekomposition}
\end{lemma}

\begin{lemma}
Let $N,l\in\NN$, $n\geq N$ and $i_{1}<...<i_{l}\leq N$,  $j_{1}<...<j_{l}\leq n$ be two ascending sequences of natural numbers. Then there is $p\in\IInc(\NN)_{N,n}$ such that $p(i_r)=j_r$ for all $r\in[l]$ if and only if $j_r-j_{r-1}\geq i_r-i_{r-1}$ for all $r\in[l+1]$, where we set $i_0=j_0=0$ and $i_{l+1}=N+1$, $j_{l+1}=n+1$.
\label{inc-lemma allgemein}
\end{lemma}
\begin{proof}
We use induction on $n\geq N$. For $n=N$, the restriction of each element from $\IInc(\NN)_{N,n}$ to $[N]$ is the identity on $[N]$. So the identities $p(i_r)=j_r$ imply $i_r=j_r$ for all $r\in[l]$ and therefore $j_{r}-j_{r-1}=i_{r}-i_{r-1}$ for all $r\in[l+1]$. Conversely, assume that $j_{r}-j_{r-1}\geq i_{r}-i_{r-1}$ holds for all $r\in[l+1]$. If one of these inequalities was strict, we would obtain:
\begin{equation*}
N+1=j_{l+1}=\sum_{r=1}^{l+1}(j_{r}-j_{r-1})>\sum_{r=1}^{l+1}(i_{r}-i_{r-1})=i_{l+1}=N+1,
\end{equation*}
which is a contradiction. We conclude that $j_{r}-j_{r-1}=i_{r}-i_{r-1}$ for all $r\in[l+1]$, so $j_{r}=i_{r}=\iid_{\NN}(i_{r})$ for all $r\in[l]$.\newline
Now assume that our claim holds for $n$. Let $j_1<...<j_l\leq n+1$ and $p\in\IInc(\NN)_{N,n+1}$ with $j_r=p(i_r)$ for all $r\in[l]$. By Lemma \ref{dekomposition}, $p$ has a decomposition $p=p_{2}\circ p_{1}$ with $p_{1}\in\IInc(\NN)_{N,n}$ and $p_{2}\in\IInc(\NN)_{n,n+1}$. Let $k_r:=p_1(i_r)$ for all $r\in[l]$. As $\lbrace k_{1},...,k_{l}\rbrace$ is a subset of $[n]$, there is either $s\in[l]$ with $p_2(k_r)=k_r$ for all $r<s$ and $p_2(k_r)=k_r+1$ for all $r\geq s$ or the restriction of $p_2$ to $[k_l]$ is the identity on $[k_l]$. Setting $s:=l+1$ in the second case and letting $k_0:=0$, $k_{l+1}:=n+1$, we obtain for both cases, for every $r\in[l+1]$:
\begin{equation*}
j_{r}-j_{r-1}=\left\lbrace\begin{array}{ll}
k_{r}-k_{r-1}&,\,r<s\text{ or }r>s\\
k_{r}-k_{r-1}+1&,\,r=s
\end{array}\right.,
\end{equation*}
hence, by induction, $j_{r}-j_{r-1}\geq i_{r}-i_{r-1}$.\newline
Now let $j_{1}<...<j_{l}\leq n+1$ be a sequence of natural numbers satisfying $j_{r}-j_{r-1}\geq i_{r}-i_{r-1}$ for all $r\in[l+1]$. We have
\begin{equation*}
\sum_{r=1}^{l+1}\left(j_{r}-j_{r-1}\right)-\left(i_{r}-i_{r-1}\right)=j_{l+1}-i_{l+1}=\left(n+2\right)-\left(N+1\right)\geq 1,
\end{equation*}
so there is $s\in[l+1]$ with $j_{s}-j_{s-1}>i_{s}-i_{s-1}\geq 1$. Define $p\in\IInc(\NN)$ by
\begin{equation*}
p(k):=\left\lbrace\begin{array}{ll}
k&,k\leq j_{s}-2\\
k+1&,k\geq j_{s}-1
\end{array}\right..
\end{equation*}
By the choice of $s$, we have $j_{s-1}\leq j_s-2$, so $j_1,...,j_{l+1}$ are contained in the image of $p$ and we can define the sequence $k_1<...<k_{l+1}$ by setting $k_{r}:=p^{-1}(j_{r})$. For $r\in[l+1]$, we have
\begin{equation*}
k_{r}=\left\lbrace\begin{array}{ll}
j_{r}&,\,r\leq s-1\\
j_{r}-1&,\,r\geq s
\end{array}\right..
\end{equation*}
In particular, $l+1\geq s$ implies $k_{l+1}=j_{l+1}-1=n+1$. Setting $k_0:=0$, we obtain
\begin{equation*}
(k_{r}-k_{r-1})-(i_{r}-i_{r-1})=\left\lbrace\begin{array}{ll}
(j_{r}-j_{r-1})-(i_{r}-i_{r-1})&,\,r\neq s\\
(j_{r}-j_{r-1})-(i_{r}-i_{r-1})-1&,\,r=s
\end{array}\right.
\end{equation*}
for all $r\in[l+1]$ and thus $k_{r}-k_{r-1}\geq i_{r}-i_{r-1}$ by the choice of $s$. Hence by induction, there is $q\in\IInc(\NN)_{N,n}$ with $q(i_{r})=k_{r}$ for all $r\in[l]$, yielding for every $r\in[l]$ the identity $j_{r}=(p\circ q)(i_{r})$. By Lemma \ref{dekomposition}, $p\circ q$ is contained in $\IInc(\NN)_{N,n+1}$, so the claim follows for $n+1$.
\end{proof}

\begin{lemma}
Let $m\leq n$ be natural numbers, $f\in R_m$ a polynomial of degree $\deg(f)>0$ and $p\in\IInc(\NN)$ with $p\cdot f\in R_{n}$. Let $m',n'\in\mathbb{N}$ be minimal with $f\in R_{m'}$ and $p\cdot f\in R_{n'}$. Then the following equivalence holds: There is $p'\in\IInc(\NN)_{m,n}$ such that $p'\cdot f=p\cdot f$ if and only if $m-m'\leq n-n'$.
\label{ersetzen durch abbildung aus Pi_n,m}
\end{lemma}
\begin{proof}
Let $i_{1}<...<i_{l}\leq m$ be the indices for which there is $k_r\in[c]$ such that $f$ contains the variable $x_{k_{r},i_{r}}$, and let $j_{1}<...<j_{l}\leq n$ be the corresponding indices for $p\cdot f$. By assumption, we have $l\geq 1$, $i_{l}=m'$ and $j_{l}=n'$. As $p(i_{r})=j_{r}$ for all $r\in[l]$, we have $p\in\IInc(\NN)_{m',n'}$ and by Lemma \ref{inc-lemma allgemein} we obtain:
\begin{equation*}
j_{r}-j_{r-1}\geq i_{r}-i_{r-1}~\text{for all }r\in[l+1]
\end{equation*}
with $i_{0}=j_{0}=0$ and $i_{l+1}=m'+1$, $j_{l+1}=n'+1$. So again by Lemma \ref{inc-lemma allgemein} there is $p'\in\IInc(\NN)_{m,n}$ with $p'(i_{r})=j_{r}$ for all $r\in[l]$ if and only if $n+1-n'=n+1-j_{l}\geq m+1-i_{l}=m+1-m'$.
\end{proof}

\begin{lemma}
Let $i_1\leq i_2\leq...$ be an ascending sequence of natural numbers and $g_{i_n}\in R_{i_n}$ be monomials. Then there are indices $j<k$ such that $g_{i_k}$ is contained in $\langle\IInc(\NN)_{i_j,i_k}\cdot g_{i_j}\rangle_{R_{i_k}}$.
\label{c mal NN: stabilit"at monomialer idealketten}
\end{lemma}
\begin{proof}
There is nothing to show if $g_{i_n}\in\KK$ for some $n\in\NN$, so assume $\deg(g_{i_n})>0$ for all $n$. By Theorem 3.1 in \cite{hs}, there is an infinite subsequence $(g_{i_{n_k}})_{k\geq 1}$ of $(g_{i_n})_{n\geq 1}$ such that for each $k\in\NN$ we have $g_{i_{n_{k+1}}}=f_k(p_k\cdot g_{i_{n_k}})$ for a monomial $f_k\in R_{i_{n_{k+1}}}$ and $p_k\in\IInc(\NN)$. We claim that one of the $p_k$ can be substituted for an element from $\IInc(\NN)_{i_{n_k},i_{n_{k+1}}}$. By contradiction, assume that this is not the case. For each $k$ let $m_k\leq i_{n_k}$ be minimal with $g_{i_{n_k}}\in R_{m_k}$. By Lemma \ref{ersetzen durch abbildung aus Pi_n,m} we have $i_{n_k}-m_k>i_{n_{k+1}}-m_{k+1}$ for every $k\geq 1$. But this contradicts the fact that there are no infinite, strictly decreasing sequences of natural numbers.
\end{proof}

We now return to our problem of stability indices of initial ideal chains with respect to $\IInc(\NN)$-compatible term orders. We begin with the remark that if $J_\circ$ is an $\IInc(\NN)$-invariant ideal chain, then every $N\geq\IInd(J_\circ)$ satisfies the stability condition (\ref{intro: stability index}).

\begin{lemma}
Let $J_\circ$ be an $\IInc(\NN)$-invariant ideal chain in $R$ and let $N\geq\IInd(J_\circ)$. Then
\begin{equation*}
\left\langle\IInc(\NN)_{N,n}\cdot J_N\right\rangle_{R_n}=J_n
\end{equation*}
for all $n\geq N$.
\label{weitere aspekte: gr"o"sere indizes erf"ullen stabilit"atsbedingung}
\end{lemma}
\begin{proof}
Let $N\geq\IInd(J_\circ)$. Then by Lemma \ref{dekomposition} and the $\IInc(\NN)$-invariance of $J_\circ$, we have
\begin{align*}
J_{n}&=\left\langle\IInc(\NN)_{\IInd(J_{\circ}),n}\cdot J_{\IInd(J_{\circ})}\right\rangle_{R_{n}}\\
&=\left\langle\IInc(\NN)_{N,n}\cdot\left(\IInc(\NN)_{\IInd(J_{\circ}),N}\cdot J_{\IInd(J_{\circ})}\right)\right\rangle_{R_{n}}\subseteq\left\langle\IInc(\NN)_{N,n}\cdot J_{N}\right\rangle_{R_{n}}.
\end{align*}
\end{proof}

The key to our proof of Theorem \ref{theorem: boundedness of stability indices} is the following proposition.

\begin{proposition}
Let $J_{\circ}=J_{1}\subseteq J_{2}\subseteq...$ be an $\IInc(\NN)$-invariant ideal chain in $R$ and $N\geq \IInd(J_\circ)$. Then for every $\IInc(\NN)$-compatible term order $\preceq$, the identity $\iin_{\preceq}(J_{2N})=\langle\IInc(\NN)_{N,2N}\cdot\iin_{\preceq}(J_N)\rangle_{R_{2N}}$ implies that $\IInd(\iin_\preceq(J_\circ))\leq 2N$.
\label{konstanzphasen}
\end{proposition}
\begin{proof}
Let $N\geq\IInd(J_\circ)$ and $\preceq$ be an $\IInc(\NN)$-compatible term order. Suppose that $\iin_{\preceq}(J_{2N})=\langle\IInc(\NN)_{N,2N}\cdot\iin_{\preceq}(J_N)\rangle_{R_{2N}}$. To prove the proposition, it is enough to show that the corresponding identity holds for every $n>2N$, as this implies
\begin{align*}
\iin_\preceq(J_n)&=\left\langle\IInc(\NN)_{N,n}\cdot\iin_\preceq(J_N)\right\rangle_{R_n}\\
&=\left\langle\IInc(\NN)_{2N,n}\cdot(\IInc(\NN)_{N,2N}\cdot\iin_\preceq(J_N))\right\rangle_{R_n}\\
&\subseteq\left\langle\IInc(\NN)_{2N,n}\cdot\iin_\preceq(J_{2N})\right\rangle_{R_n},
\end{align*}
where we used Lemma \ref{dekomposition} in the second and the $\IInc(\NN)$-invariance of $\iin_\preceq(J_\circ)$ in the third line. To this end, it suffices to show that if $\mathcal{G}$ is a Gr\"obner basis of $J_N$ with respect to $\preceq$, then $\mathcal{G}':=\IInc(\NN)_{N,n}\cdot\mathcal{G}$ is a Gr\"obner basis of $J_{n}$ with respect to $\preceq$, because this in turn implies
\begin{align*}
\iin_{\preceq}(J_{n})&=\left\langle\iin_{\preceq}(g')\,|\,g'\in\mathcal{G}'\right\rangle_{R_{n}}\\
&=\left\langle\IInc(\NN)_{N,n}\cdot\iin_{\preceq}(g)\,|\,g\in\mathcal{G}\right\rangle_{R_{n}}\\
&\subseteq\left\langle\IInc(\NN)_{N,n}\cdot\iin_{\preceq}(J_{N})\right\rangle_{R_{n}},
\end{align*}
where the $\IInc(\NN)$-compatibility of $\preceq$ guarantees the validity of the second identity.\newline
So let $n>2N$. As $\mathcal{G}$ generates $J_{N}$ and $N\geq\IInd(J_\circ)$, $\mathcal{G}'$ is a generating set for $J_n$ by Lemma \ref{weitere aspekte: gr"o"sere indizes erf"ullen stabilit"atsbedingung}. Thus, we only have to show that the $S$-polynomials of the elements of $\mathcal{G}'$ reduce to zero with respect to $\mathcal{G}'$. Choose $f',g'\in\mathcal{G}'$ and write $f'=p_1\cdot f$, $g'=p_2\cdot g$ with $p_1,p_2\in\IInc(\NN)_{N,n}$ and $f,g\in\mathcal{G}$. Let $j_1<...<j_{N}\leq n$, $k_{1}<...<k_{N}\leq n$ be natural numbers satisfying $p_{1}([N])=\left\lbrace j_{1},...,j_{N}\right\rbrace$ and $p_{2}([N])=\left\lbrace k_{1},...,k_{N}\right\rbrace$ and let $i_{1}<...<i_{2N}\leq n$ be natural numbers such that $\left\lbrace j_{1},...,j_{N}\right\rbrace\cup\left\lbrace k_{1},...,k_{N}\right\rbrace\subseteq\left\lbrace i_{1},...,i_{2N}\right\rbrace$. Define the map $p$ by
\begin{equation*}
p(j):=\left\lbrace\begin{array}{ll}
i_{j}&,j\in[2N]\\
n+j&,j>2N
\end{array}\right..
\end{equation*}
Then $p$ is an element of $\IInc(\NN)_{2N,n}$ satisfying $p_{1}([N]),p_{2}([N])\subseteq p(\NN)$. We first want to show that $p^{-1}\cdot f'$ and $p^{-1}\cdot g'$ lie in $J_{2N}$. Due to the $\IInc(\NN)$-invariance of $J_\circ$, this can be achieved by proving that the maps $(p^{-1}\circ p_{1})_{|[N]}$ and $(p^{-1}\circ p_{2})_{|[N]}$ can be extended to elements from $\IInc(\NN)_{N,2N}$. As $p_1$, $p_2$ and $(p_{|p(\NN)})^{-1}$ are strictly increasing, the same is true for the restrictions of $p^{-1}\circ p_{1}$ and $p^{-1}\circ p_{2}$ to $[N]$. Furthermore, we have
\begin{equation*}
p^{-1}(p_{1}(N))\leq p^{-1}(i_{2N})=2N
\end{equation*}
and the analogous inequality holds for $p^{-1}(p_{2}(N))$. Thus, defining $q_1$ and $q_2$ as
\begin{equation*}
q_{1}(i):=\left\lbrace\begin{array}{ll}
(p^{-1}\circ p_{1})(i)&,i\in[N]\\
2N+i&,i>N
\end{array}\right.,~
q_{2}(i):=\left\lbrace\begin{array}{ll}
(p^{-1}\circ p_{2})(i)&,i\in[N]\\
2N+i&,i>N
\end{array}\right.
\end{equation*}
yields the desired extensions, and we conclude $p^{-1}\cdot f',p^{-1}\cdot g'\in J_{2N}$.\newline
Recall that for polynomials $h_1,h_2\in R$, the $S$-polynomial $S(h_1,h_2)$ of $h_1$ and $h_2$ with respect to the term order $\preceq$ is defined as
\begin{equation*}
S(h_1,h_2)=\lcm(\iin_{\preceq}(h_1),\iin_{\preceq}(h_2))\left(\frac{h_1}{\lt_{\preceq}(h_1)}-\frac{h_2}{\lt_{\preceq}(h_2)}\right),
\end{equation*}
where $\lcm(\iin_{\preceq}(h_1),\iin_{\preceq}(h_2))$ stands for the least common multiple of $\iin_{\preceq}(h_1)$ and $\iin_{\preceq}(h_2)$ and $\lt_\preceq(h_i)$ denotes the leading term of $h_i$, i.e. the product of the leading monomial of $h_i$ with respect to $\preceq$ and its coefficient in $h_i$. Due to the $\IInc(\NN)$-compatibility of $\preceq$, the $S$-polynomial of $p^{-1}\cdot f'$ and $p^{-1}\cdot g'$ satisfies
\begin{align}
\nonumber S(p^{-1}\cdot f',p^{-1}\cdot g')&=\lcm(\iin_{\preceq}(p^{-1}\cdot f'),\iin_{\preceq}(p^{-1}\cdot g'))\left(\frac{p^{-1}\cdot f'}{\lt_{\preceq}(p^{-1}\cdot f')}-\frac{p^{-1}\cdot g'}{\lt_{\preceq}(p^{-1}\cdot g')}\right)\\
\nonumber&=\lcm(p^{-1}\cdot\iin_{\preceq}(f'),p^{-1}\cdot\iin_{\preceq}(g'))\left(\frac{p^{-1}\cdot f'}{p^{-1}\cdot\lt_{\preceq}(f')}-\frac{p^{-1}\cdot g'}{p^{-1}\cdot\lt_{\preceq}(g')}\right)\\
\nonumber&=p^{-1}\cdot\left[\lcm(\iin_{\preceq}(f'),\iin_{\preceq}(g'))\left(\frac{f'}{\lt_{\preceq}(f')}-\frac{g'}{\lt_{\preceq}(g')}\right)\right]\\
&=p^{-1}\cdot S(f',g').\label{identity of s-polynomials}
\end{align}
As both $p^{-1}\cdot f'$ and $p^{-1}\cdot g'$ are contained in $J_{2N}$, this is also true for $S(p^{-1}\cdot f',p^{-1}\cdot g')$. Furthermore, by assumption and the $\IInc(\NN)$-compatibility of $\preceq$, the set $\IInc(\NN)_{N,2N}\cdot\mathcal{G}$ is a Gr\"obner basis of $J_{2N}$ with respect to $\preceq$. Therefore, $S(p^{-1}\cdot f',p^{-1}\cdot g')$ reduces to zero with respect to $\IInc(\NN)_{N,2N}\cdot\mathcal{G}$, i.e. it can be written as
\begin{equation*}
S(p^{-1}\cdot f',p^{-1}\cdot g')=\sum_{i=1}^{r}h_{i}\left(q'_{i}\cdot g_{i}\right)
\end{equation*}
with $h_{i}\in R_{2N}$, $q'_{i}\in\IInc(\NN)_{N,2N}$, $g_{i}\in\mathcal{G}$ and $\iin_{\preceq}(S(p^{-1}\cdot f',p^{-1}\cdot g'))\succeq\iin_{\preceq}(h_{i}(q'_{i}\cdot g_{i}))$ for all $i\in[r]$. This yields
\begin{gather*}
\iin_{\preceq}(p\cdot S(p^{-1}\cdot f',p^{-1}\cdot g'))=p\cdot\iin_{\preceq}(S(p^{-1}\cdot f',p^{-1}\cdot g'))\\
\succeq p\cdot\iin_{\preceq}(h_{i}(q'_{i}\cdot g_{i}))=\iin_{\preceq}((p\cdot h_{i})((p\circ q'_{i})\cdot g_{i}))
\end{gather*}
for all $i\in[r]$. As by equation (\ref{identity of s-polynomials}), we have
\begin{equation*}
S(f',g')=p\cdot S(p^{-1}\cdot f',p^{-1}\cdot g')=\sum_{i=1}^{r}(p\cdot h_{i})((p\circ q'_{i})\cdot g_{i}),
\end{equation*}
and $p\circ q_i'\in\IInc(\NN)_{N,n}$ by Lemma \ref{dekomposition}, we conclude that $S(f',g')$ reduces to zero with respect to $\mathcal{G}'$.
\end{proof}

%-------------------------------------------------------------------------------------------------------------%
%-------------------------------------------------------------------------------------------------------------%
\section{Proof of Theorem \ref{theorem: boundedness of stability indices} and implications}\label{section 1}

\begin{proof}[Proof of Theorem \ref{theorem: boundedness of stability indices}]
By contradiction, assume the existence of a sequence $(\preceq_{n})_{n\geq 1}$ of $\IInc(\NN)$-compatible term orders on $R$ with $\lim_{n\rightarrow\infty}\IInd(\iin_{\preceq_{n}}(J_{\circ}))=\infty$. Set $N_0:=\IInd(J_\circ)$ and $N_i:=2N_{i-1}$ for $i\geq 1$. We claim that there is a collection $(\preceq^i_n)_{n\geq 1}$ of infinite subsequences of $(\preceq_n)_{n\geq 1}$, where $i$ ranges over $\NN\cup\lbrace 0\rbrace$, such that
\begin{compactenum}[(1)]
\item $(\preceq^{i}_n)_{n\geq 1}$ is a subsequence of $(\preceq^{i-1}_n)_{n\geq 1}$ for all $i\geq 1$;
\item $\iin_{\preceq_n^{i}}(J_{N_{i}})\subsetneq\langle\IInc(\NN)_{N_{i-1},N_i}\cdot\iin_{\preceq^i_n}(J_{N_{i-1}})\rangle_{R_{N_i}}$ for all $i,n\geq 1$;
\item $\iin_{\preceq_n^{i}}(J_{N_{i}})=\iin_{\preceq_1^{i}}(J_{N_{i}})$ for all $i,n\geq 1$.
\end{compactenum}
Indeed, we can construct these subsequences as follows: Set $(\preceq^0_n)_{n\geq 1}:=(\preceq_n)_{n\geq 1}$. By induction, assume that the subsequence $(\preceq^i_n)_{n\geq 1}$ has already been defined for some $i\geq 0$. Then $\lim_{n\rightarrow\infty}\IInd(\iin_{\preceq^i_{n}}(J_{\circ}))=\infty$, so in particular, there are infinitely many indices $n$ such that $\IInd(\iin_{\preceq^i_{n}}(J_{\circ}))>N_{i+1}$. By Proposition \ref{konstanzphasen}, these indices satisfy $\iin_{\preceq_n^i}(J_{N_{i+1}})\subsetneq\langle\IInc(\NN)_{N_{i},N_{i+1}}\cdot\iin_{\preceq^i_n}(J_{N_{i}})\rangle_{R_{N_{i+1}}}$. Hence, we obtain an infinite subsequence of $(\preceq^i_n)_{n\geq 1}$ satisfying (2). As the total number of initial ideals of $J_{N_{i+1}}$ is finite, this subsequence contains another infinite subsequence $(\preceq^{i+1}_n)_{n\geq 1}$ such that $\iin_{\preceq^{i+1}_n}(J_{N_{i+1}})=\iin_{\preceq^{i+1}_1}(J_{N_{i+1}})$ for all $n$ and we are done.\newline
For every $i\geq 1$, choose a monomial $g_i\in\iin_{\preceq^i_1}(J_{N_i})$ that is not contained in $\langle\IInc(\NN)_{N_{i-1},N_i}\cdot\iin_{\preceq^i_1}(J_{N_{i-1}})\rangle_{R_{N_i}}$. Then for any pair $i<j$ of natural numbers, we have
\begin{align*}
g_j&\not\in\left\langle\IInc(\NN)_{N_{j-1},N_j}\cdot\iin_{\preceq^j_1}(J_{N_{j-1}})\right\rangle_{R_{N_j}}\\
&\supseteq\left\langle\IInc(\NN)_{N_{j-1},N_j}\cdot(\IInc(\NN)_{N_i,N_{j-1}}\cdot\iin_{\preceq^j_1}(J_{N_{i}}))\right\rangle_{R_{N_j}}\\
&=\left\langle\IInc(\NN)_{N_{j-1},N_j}\cdot(\IInc(\NN)_{N_i,N_{j-1}}\cdot\iin_{\preceq^i_1}(J_{N_{i}}))\right\rangle_{R_{N_j}}\\
&=\left\langle\IInc(\NN)_{N_i,N_j}\cdot\iin_{\preceq^i_1}(J_{N_{i}})\right\rangle_{R_{N_j}}\\
&\supseteq\left\langle\IInc(\NN)_{N_i,N_j}\cdot g_i\right\rangle_{R_{N_j}},
\end{align*}
where we used properties (1) and (3) in the third and Lemma \ref{dekomposition} in the fourth line. But by Lemma \ref{c mal NN: stabilit"at monomialer idealketten}, such a sequence $(g_i)_{i\geq 1}$ cannot exist, and we have arrived at a contradiction.
\end{proof}

\begin{bemerkung}
The global stability index $\IInd(J_\circ)$ of an $\IInc(\NN)$-invariant ideal chain can be smaller than $\max(\mathrm{I}(J_\circ))$: Let $c=1$, $J_1=J_2=J_3=\lbrace 0\rbrace$, $J_4=\langle x_1+x_3\rangle_{R_4}$ and $J_n=\langle\IInc(\NN)_{4,n}\cdot J_4\rangle_{R_n}$ for $n\geq 5$. Let $\preceq$ be any $\IInc(\NN)$-compatible term order satisfying $x_n\preceq x_{n+1}$ for all $n\in\NN$. As $(x_2+x_4)-(x_1+x_4)=x_2-x_1$ lies in $J_5$, we conclude that $x_2\in\iin_\preceq(J_5)$. On the other hand, we have $\iin_\preceq(J_4)=\langle x_3\rangle_{R_4}$. Thus, $x_2$ is not an element of $\langle\IInc(\NN)_{4,5}\cdot\iin_{\preceq}(J_4)\rangle_{R_5}$, so $\IInd(\iin_\preceq(J_\circ))>4=\IInd(J_\circ)$.
\label{example: global index smaller than maximum}
\end{bemerkung}

%-------------------------------------------------------------------------------------------------------------%
%-------------------------------------------------------------------------------------------------------------%
We next want to study some of the consequences of Theorem \ref{theorem: boundedness of stability indices}, which include the statements on finiteness of numbers of initial ideals and initial ideal chains described in the introduction of this article. To this end, we need a few more preparations. Setting
\begin{equation*}
S_n:=\left\lbrace\sigma:\NN\rightarrow\NN\,|\,\sigma\text{ is bijective, }\sigma(i)=i\text{ for all }i\geq n+1\right\rbrace
\end{equation*}
and $S_\infty:=\bigcup_{n\geq 1}S_n$, we can define an action of $S_\infty$ on $R$ by $\KK$-linear extension of the maps
\begin{equation*}
\sigma\cdot(x_{i_1,j_1}^{e_1}\cdot...\cdot x_{i_r,j_r}^{e_r}):=x_{i_1,\sigma(j_1)}^{e_1}\cdot...\cdot x_{i_r,\sigma(j_r)}^{e_r}
\end{equation*}
for every $\sigma\in S_\infty$. There is the following inclusion of orbits:

\begin{lemma}[cf. \cite{nr}, Lemma 7.5]
For every pair of natural numbers $m\leq n$ and $f\in R_m$, we have $\IInc(\NN)_{m,n}\cdot f\subseteq S_n\cdot f$.
\label{c mal NN: inklusion der bahnen}
\end{lemma}

Lemma \ref{c mal NN: inklusion der bahnen} ensures that every $S_\infty$-invariant ideal chain $J_\circ=J_1\subseteq J_2\subseteq...$ in $R$, i.e. every ideal chain satisfying $S_n\cdot J_m\subseteq J_n$ for all $m\leq n$, is also $\IInc(\NN)$-invariant. Note that the ideals $J_n$ of an $S_\infty$-invariant ideal chain are themselves $S_n$-invariant, i.e. they satisfy $S_n\cdot J_n\subseteq J_n$.\newline
For any subset $A\subseteq\NN$, let $R_A$ be the polynomial ring over $\KK$ in the variables indexed by $[c]\times A$.

\begin{lemma}
Let $J\subseteq R_n$ be an ideal satisfying $S_n\cdot J\subseteq J$, $m\leq n$ and $p\in\IInc(\NN)_{m,n}$. Then
\begin{equation*}
p\cdot\left(J\cap R_{m}\right)=J\cap R_{p([m])}.
\end{equation*}
In particular, for every $f\in J\cap R_m$, the polynomial $p\cdot f$ is contained in $J$.
\label{weitere aspekte: inc-produkte liegen in ideal}
\end{lemma}
\begin{proof}
The inclusion $p\cdot\left(J\cap R_{m}\right)\subseteq J\cap R_{p([m])}$ follows from the $S_n$-invariance of $J$ and from Lemma \ref{c mal NN: inklusion der bahnen}. Conversely, let $f\in J\cap R_{p([m])}$ and $\sigma\in S_n$ satisfying $\sigma_{|[m]}=p_{|[m]}$. Then $\sigma^{-1}_{|p([m])}=p^{-1}_{|p([m])}$, and due to the $S_n$-invariance of $J$ we obtain
\begin{equation*}
f=p\cdot\left(p^{-1}\cdot f\right)=p\cdot\left(\sigma^{-1}\cdot f\right)\in p\cdot\left(J\cap R_{m}\right).
\belowdisplayskip=0pt
\end{equation*}
\end{proof}

\begin{lemma}
Let $n\in\NN$ and $J\subseteq R_n$ be an ideal. Then for any $\IInc(\NN)$-compatible term order $\preceq$, the following identity holds:
\begin{equation}
p\cdot\iin_\preceq(J)=\iin_\preceq(p\cdot J).
\label{equation: commutation of Inc and in}
\end{equation}
In this equation, $p\cdot J$ is regarded as an ideal of $R_{p([n])}$.
\label{lemma: Inc and formation of initial ideals commute}
\end{lemma}
\begin{proof}
The left side of equation (\ref{equation: commutation of Inc and in}) is generated, as an ideal of $R_{p([n])}$, by the set $G_1:=\lbrace p\cdot\iin_\preceq(f)\,|\,f\in J\rbrace$, whereas the right side is generated by $G_2:=\lbrace\iin_\preceq(p\cdot f)\,|\,f\in J\rbrace$. The $\IInc(\NN)$-compatibility of $\preceq$ yields $G_1=G_2$, and the identity in (\ref{equation: commutation of Inc and in}) follows.
\end{proof}

\begin{satz}
For an $\IInc(\NN)$-invariant ideal chain $J_{\circ}$, the following statements are equivalent:
\begin{compactenum}[(1)]
\item $\mathrm{I}(J_\circ)$ is bounded above.
\item The set of ideal chains $\lbrace\iin_{\preceq}(J_\circ)\,|\,\preceq\text{ is an }\IInc(\NN)\text{-compatible term order on }R\rbrace$ is finite.
\end{compactenum}
Furthermore, the two above statements imply:
\begin{compactenum}[(3)]
\item There is $N\in\NN$ such that
\begin{equation}
\iin_{\preceq}(J_{n})=\sum_{1\leq i_{1}<...<i_{N}\leq n}\left\langle\iin_{\preceq}(J_{n}\cap R_{\lbrace i_{1},...,i_{N}\rbrace})\right\rangle_{R_{n}}
\label{equation: initial ideals generated by initial ideals of intersections}
\end{equation}
for all $n\geq N$ and every $\IInc(\NN)$-compatible term order $\preceq$ on $R$. Here, we regard the intersections $J_n\cap R_{\lbrace i_{1},...,i_{N}\rbrace}$ as ideals of $R_{\lbrace i_{1},...,i_{N}\rbrace}$.
\end{compactenum}
If $J_\circ$ is not only $\IInc(\NN)$- but also $S_\infty$-invariant then (3) is equivalent to (1) and (2).
\label{weitere aspekte: charakterisierungen}
\end{satz}
\begin{proof}
We first show the equivalence of (1) and (2) for $\IInc(\NN)$-invariant ideal chains $J_{\circ}$. The implication (2) $\Rightarrow$ (1) is clear. For the reverse implication, let $N:=\max(\mathrm{I}(J_\circ))$. Then by Lemma \ref{weitere aspekte: gr"o"sere indizes erf"ullen stabilit"atsbedingung}, for $\IInc(\NN)$-compatible term orders $\preceq,\preceq'$ we have $\iin_\preceq(J_\circ)=\iin_{\preceq'}(J_\circ)$ if and only if $\iin_\preceq(J_n)=\iin_{\preceq'}(J_n)$ for all $n\in[N]$. As $J_1,....,J_N$ each have only finitely many initial ideals, there are only finitely many sequences $L_1\subseteq...\subseteq L_N$ such that $L_n=\iin_{\preceq}(J_n)$ for all $n$ for any term order $\preceq$ on $R$. Hence, assertion (2) follows.\\
Next, we show the implication (1) $\Rightarrow$ (3) for $\IInc(\NN)$-invariant $J_\circ$. Let again $N:=\max(\mathrm{I}(J_\circ))$ and choose any $\IInc(\NN)$-compatible term order $\preceq$ on $R$. Then by Remark \ref{weitere aspekte: gr"o"sere indizes erf"ullen stabilit"atsbedingung} and the $\IInc(\NN)$-compatibility of $\preceq$, $\iin_{\preceq}(J_n)$ is generated by
\begin{equation*}
\left\lbrace\iin_{\preceq}(p\cdot f)\,|\,p\in\IInc(\NN)_{N,n},f\in J_N\right\rbrace.
\end{equation*}
As $J_\circ$ is $\IInc(\NN)$-invariant, each of the polynomials $p\cdot f$ in the above set lies in one of the intersections $J_n\cap R_{\lbrace i_1,...,i_N\rbrace}$, where $i_1<...<i_N$ ranges over all strictly ascending sequences of $[n]$. This proves the inclusion $\subseteq$ in (\ref{equation: initial ideals generated by initial ideals of intersections}). The reverse inclusion is obvious.\newline
Now assume $J_\circ$ to be $S_\infty$-invariant and that (3) holds. By the Noetherianity of $R_N$, there is an index $N'\geq N$ such that $J_n\cap R_N=J_{N'}\cap R_N=:J$ for all $n\geq N'$. Let $n\geq N'$ and $\preceq$ be any $\IInc(\NN)$-compatible term order on $R$. For a sequence $1\leq i_1<...<i_N\leq n$, let $p_{i_1,...,i_N}\in\IInc(\NN)$ be any function satisfying $p_{i_1,...,i_N}([N])=\lbrace i_1,...,i_N\rbrace$. Then by Lemmata \ref{weitere aspekte: inc-produkte liegen in ideal} and \ref{lemma: Inc and formation of initial ideals commute}, we have
\begin{align*}
\iin_\preceq(J_n)&=\sum_{1\leq i_1<...<i_N\leq n}\left\langle\iin_\preceq(J_n\cap R_{\lbrace i_1,...,i_N\rbrace})\right\rangle_{R_n}\\
&=\sum_{1\leq i_1<...<i_N\leq n}\left\langle\iin_\preceq(p_{i_1,...,i_N}\cdot J)\right\rangle_{R_n}\\
&=\sum_{1\leq i_1<...<i_N\leq n}\left\langle p_{i_1,...,i_N}\cdot\iin_\preceq(J)\right\rangle_{R_n}.
\end{align*}
By Lemma \ref{dekomposition}, each of the $p_{i_1,...,i_N}$ has a decomposition $p_{i_1,...,i_N}=p^{(2)}_{i_1,...,i_N}\circ p^{(1)}_{i_1,...,i_N}$ with $p^{(1)}_{i_1,...,i_N}\in\IInc(\NN)_{N,N'}$ and $p^{(2)}_{i_1,...,i_N}\in\IInc(\NN)_{N',n}$. Thus, we obtain
\begin{align*}
&\sum_{1\leq i_1<...<i_N\leq n}\langle p_{i_1,...,i_N}\cdot\iin_\preceq(J)\rangle_{R_n}\\
&=\sum_{1\leq i_1<...<i_N\leq n}\langle p^{(2)}_{i_1,...,i_N}\cdot(p^{(1)}_{i_1,...,i_N}\cdot\iin_\preceq(J))\rangle_{R_n}\\
&=\sum_{1\leq i_1<...<i_N\leq n}\langle p^{(2)}_{i_1,...,i_N}\cdot\iin_\preceq(p^{(1)}_{i_1,...,i_N}\cdot J)\rangle_{R_n}\\
&\subseteq\langle\IInc(\NN)_{N',n}\cdot\iin_\preceq(J_{N'})\rangle_{R_n},
\end{align*}
where we again used Lemma \ref{lemma: Inc and formation of initial ideals commute} for the second identity and Lemma \ref{weitere aspekte: inc-produkte liegen in ideal} for the last inclusion. This shows that $\IInd(\iin_\preceq(J_\circ))\leq N'$ for any $\IInc(\NN)$-compatible term order $\preceq$ and (1) follows.
\end{proof}

\begin{bemerkung}
By equation (\ref{equation: initial ideals generated by initial ideals of intersections}), for every $\IInc(\NN)$-invariant ideal chain $J_\circ$ in $R$, there is a natural number $N$ such that for every $n\geq N$ and every $\IInc(\NN)$-compatible term order $\preceq$ on $R$, there is a Gr\"obner basis of $J_n$ with respect to $\preceq$ whose elements each contain no more than $cN$ distinct variables. This is not the case for arbitrary ideal chains in $R$. For instance, set $c=1$ and consider the ideal chain $J_\circ$ defined by
\begin{align*}
&J_1:=\lbrace 0\rbrace,\\
&J_{2^n}:=\langle J_{2^{n-1}},x_{2^{n-1}+1}+...+x_{2^n}\rangle_{R_{2^n}}\text{ for }n\geq 1,\\
&J_m:=J_{2^n}\text{ for }2^n\leq m<2^{n+1}.
\end{align*}
Then for any term order $\preceq$ on $R$ and $n\geq 1$, every polynomial $f\in J_{2^n}$ with $\iin_\preceq(f)\mid\iin_\preceq(x_{2^{n-1}+1}+...+x_{2^n})$ must contain a non-trivial $\KK$-multiple of $x_{2^{n-1}+1}+...+x_{2^n}$ and, hence, at least $2^{n-1}$ distinct variables.
\end{bemerkung}

\begin{korollar}
Let $J$ be an ideal of $R$ satisfying $\IInc(\NN)\cdot J\subseteq J$. Then $J$ has only finitely many initial ideals with respect to $\IInc(\NN)$-compatible term orders on $R$.
\label{corollary: finitely many initial ideals}
\end{korollar}
\begin{proof}
For every term order $\preceq$ on $R$, $\iin_\preceq(J)$ is generated by the union of all initial ideals $\iin_\preceq(J\cap R_n)\subseteq R_n$. Thus, if $\preceq,\preceq'$ are term orders on $R$ with $\iin_\preceq(J\cap R_n)=\iin_{\preceq'}(J\cap R_n)$ for all $n$, then $\iin_\preceq(J)=\iin_{\preceq'}(J)$. As the ideal chain $J_\circ:=J\cap R_1\subseteq J\cap R_2\subseteq...$ is $\IInc(\NN)$-invariant, Theorem \ref{weitere aspekte: charakterisierungen}(2) tells us that there exists a finite number of $\IInc(\NN)$-compatible term orders $\preceq_1,...,\preceq_N$ on $R$ such that for every $\IInc(\NN)$-compatible term order $\preceq$ on $R$ there is $i\in[N]$ with $\iin_\preceq(J_\circ)=\iin_{\preceq_i}(J_\circ)$. This proves our claim.
\end{proof}

\begin{bemerkung}
There is a more direct way to prove Corollary \ref{corollary: finitely many initial ideals} which does not rely on Theorem \ref{weitere aspekte: charakterisierungen}. Namely, one can transfer the proof of finiteness of the number of initial ideals for ideals in polynomial rings in finitely many variables given in \cite{rob}, Lemma 2.6, to the situation of $\IInc(\NN)$-invariant ideals in $R$: Just substitute the ideals $\mathrm{m_i}$ defined in \cite{rob}  for $\langle\IInc(\NN)\cdot\mathrm{m_i}\rangle_R$ and use the fact that $\IInc(\NN)$-divisibility in $R$ is a well-partial-order (\cite{hs}, Theorem 3.1) as a substitute for Noetherianity. This raises the question whether Theorem \ref{theorem: boundedness of stability indices} is just a simple consequence of Corollary \ref{corollary: finitely many initial ideals}.\newline
Indeed, for any $\IInc(\NN)$-invariant ideal chain $J_\circ$ in $R$, the ideal $J:=\bigcup_{n\geq 1}J_n$ is an $\IInc(\NN)$-invariant ideal of $R$, and for every term order $\preceq$ on $R$, we have $\iin_\preceq(J)=\bigcup_{n\geq 1}\iin_\preceq(J_n)$. Hence, Corollary \ref{corollary: finitely many initial ideals} yields
\begin{equation*}
\#\bigg\lbrace\bigcup_{n\geq 1}\iin_\prec(J_n)\,|\,\preceq\text{ is }\IInc(\NN)\text{-compatible}\bigg\rbrace<\infty.
\end{equation*}
However, Theorem \ref{theorem: boundedness of stability indices} provides more information than that: By Theorem \ref{weitere aspekte: charakterisierungen}(2), not only the number of unions of the initial ideals of the $J_n$ with respect to $\IInc(\NN)$-compatible term orders is finite, but also the number of sequences $(\iin_\preceq(J_n))_{n\geq 1}$ giving rise to the same union.
\end{bemerkung}

\begin{bemerkung}
Corollary \ref{corollary: finitely many initial ideals} does not hold for the number of initial ideals with respect to arbitrary term orders on $R$: Let $c=1$ and $J:=\langle\IInc(\NN)\cdot (x_1^2x_2+x_1x_2^2)\rangle_R$ be the ideal that is generated by the $\IInc(\NN)$-orbits of the polynomial $x_1^2x_2+x_1x_2^2$. For every $n\in\NN$, define the term order $\preceq_n$ by
\begin{gather*}
x_{\sigma_n(1)}^{a_1}\cdot...\cdot x_{\sigma_n(k)}^{a_k}\prec_n x_{\sigma_n(1)}^{b_1}\cdot...\cdot x_{\sigma_n(k)}^{b_k}:\Leftrightarrow a_{i}<b_{i}\text{ for }i=\min\lbrace j\,|\,a_j\neq b_j\rbrace,
\end{gather*}
where the map $\sigma_n\in S_{\infty}$ is defined by
\begin{equation*}
\sigma_n(j)=\left\lbrace\begin{array}{ll}
n+1-j&,\,j\leq n\\
j&,\,j>n
\end{array}\right..
\end{equation*}
For example, if $n=3$, then $(\sigma_3(1),\sigma_3(2),\sigma_3(3),\sigma_3(4),\sigma_3(5))=(3,2,1,4,5)$. We claim that for every pair $n<n'$ of natural numbers, $x_1^2x_{n'}\in\iin_{\preceq_{n}}(J)\setminus\iin_{\preceq_{n'}}(J)$. We have $\iin_{\preceq_n}(x_1^2x_{n'}+x_1x_{n'}^2)=x_1^2x_{n'}$ as $\sigma_n^{-1}(n')=n'>n=\sigma_n^{-1}(1)$, so $x_1^2x_{n'}\in\iin_{\preceq_n}(J)$. Let $f$ be a polynomial in $J$ that contains the monomial $x_1^2x_{n'}$. We may assume $f$ to be homogeneous, so $f=\sum_{i=1}^{k}c_ip_i\cdot(x_1^2x_2+x_1x_2^2)$ with $c_i\in\KK\setminus\lbrace 0\rbrace$ and $p_i\in\IInc(\NN)$, where $p_i\cdot(x_1^2x_2+x_1x_2^2)\neq p_j\cdot(x_1^2x_2+x_1x_2^2)$ for $i\neq j$.  As $f$ contains $x_1^2x_{n'}$, there is exactly one $i$ with $p_{i}\cdot(x_1^2x_2+x_1x_2^2)=x_1^2x_{n'}+x_1x_{n'}^2$. Therefore, $f$ contains the monomial $x_1x_{n'}^2$. But $x_1^2x_{n'}\prec_{n'} x_1x_{n'}^2$, so $x_1^2x_{n'}\not\in\iin_{\preceq_{n'}}(J)$. We conclude that the initial ideals $\iin_{\preceq_n}(J)$ are pairwise distinct. Thus, $J$ has infinitely many distinct initial ideals with respect to arbitrary term orders.
\label{example: infinitely many initial ideals with respect to arbitrary term orders}
\end{bemerkung}

%-------------------------------------------------------------------------------------------------------------%
%-------------------------------------------------------------------------------------------------------------%
\section{Classification of $\IInc(\NN)$-compatible monomial preorders for $c=1$}\label{section 2}
In this section, we will always assume $c=1$. Following the definition in \cite{kemp}, we call a strict partial order $\prec$ on $R$ or $R_n$ a monomial preorder if it satisfies the following conditions:
\begin{itemize}
\item Multiplicativity: For monomials $f,g,h\in R$ or $R_n$, $f\prec g$ implies $hf\prec hg$.
\item Cancellativeness: For monomials $f,g,h\in R$ or $R_n$, $hf\prec hg$ implies $f\prec g$.
\item Incomparability with respect to $\prec$ is transitive.
\end{itemize}

For every monomial $f\in R_n$ there is $a=(a_1,...,a_n)\in\NN_0^n$ with $f=x^a:=x_1^{a_1}\cdots x_n^{a_n}$. In \cite{kemp} it was shown that for every monomial preorder $\prec$ on $R_n$, there is some $m\in\NN$ and a matrix $M\in\RR^{m\times n}$ such that for monomials $x^a,x^b\in R_n$ we have
\begin{equation*}
x^a\prec x^b\Leftrightarrow M\cdot a<_{\text{lex}}M\cdot b,
\end{equation*}
where $<_{\text{lex}}$ denotes the lexicographic order on $\RR^n$, i.e.
\begin{equation*}
(\lambda_1,...,\lambda_n)<_{\text{lex}}(\mu_1,...,\mu_n)\Leftrightarrow \lambda_{i}<\mu_{i}\text{ for }i=\min\lbrace j\,|\,\lambda_j\neq\mu_j\rbrace.
\end{equation*}
Obviously, one can assume the rows of $M$ to be orthogonal and non-zero (and, consequently, $m\leq n$), and we will do so from now on.

Our goal for this section is to classify the $\IInc(\NN)$-compatible monomial preorders on $R$, i.e. the monomial preorders $\prec$ which additionally satisfy the condition
\begin{equation}
f\prec g\Rightarrow p\cdot f\prec p\cdot g
\label{definition IInc(NN)-compatibility}
\end{equation}
for all monomials $f,g\in R$ and every $p\in\IInc(\NN)$. Our strategy is to first classify the $\IInc(\NN)$-compatible monomial preorders on $R_4$ (the question why we have to use $n=4$ is addressed in Remark \ref{remark: why R_4?}). By shifting variable indices and using Equation (\ref{definition IInc(NN)-compatibility}), we will then be able to deduce from this what the $\IInc(\NN)$-compatible monomial preorders on $R$ are.

\begin{lemma}
Let $M\in\RR^{m\times 4}$ be a matrix representing an $\IInc(\NN)$-compatible monomial preorder $\prec$ on $R_4$. Then there is a real number $\lambda\neq 0$ such that the first row $r_1\in\RR^4$ of $M$ satisfies
\begin{equation}
r_1\in\left\lbrace(\lambda,\lambda,\lambda,\lambda),(\lambda,0,0,0),(0,0,0,\lambda)\right\rbrace.
\label{first row}
\end{equation}
If $r_1=(\lambda,\lambda,\lambda,\lambda)$ and $m\geq 2$, then the second row $r_2$ of $M$ satisfies
\begin{equation}
r_2\in\left\lbrace(-\mu,-\mu,-\mu,3\mu),(3\mu,-\mu,-\mu,-\mu)\right\rbrace
\label{second row}
\end{equation}
for some $\mu\neq 0$.
\end{lemma}
\begin{proof}
Write $r_1=(a_1,a_2,a_3,a_4)$. Due to the $\IInc(\NN)$-compatibility of $\prec$, for any vector $(v_1,v_2,v_3)\in\ZZ^3$ the inequality $a_1v_1+a_2v_2+a_3v_3>0$ implies $a_2v_1+a_3v_2+a_4v_3\geq 0$ and $a_1v_1+a_3v_2+a_4v_3\geq 0$. Therefore, the matrices
\begin{equation*}
A_1:=\left(\begin{array}{ccc}
a_1&a_2&a_3\\
a_2&a_3&a_4
\end{array}\right),~~
A_2:=\left(\begin{array}{ccc}
a_1&a_2&a_3\\
a_1&a_3&a_4
\end{array}\right)
\end{equation*}
must have rank $\leq 1$. Assume $a_1\neq 0$. Then from $\rank(A_2)\leq 1$ we deduce $a_2=a_3=a_4$, which due to $\rank(A_1)\leq 1$ implies either $a_2=a_3=a_4=0$ or $a_1=a_2$. On the other hand, if $a_1=0$, $\rank(A_1)\leq 1$ implies $a_2=a_3=0$. This proves (\ref{first row}).\newline
Now assume that $r_1=(\lambda,\lambda,\lambda,\lambda)$ and $m\geq 2$. Write $r_2=(b_1,b_2,b_3,b_4)$. Again, the $\IInc(\NN)$-compatibility of $\prec$ implies that if $(v_1,v_2,-v_1-v_2)\in\ZZ^3$ satisfies $b_1v_1+b_2v_2+b_3(-v_1-v_2)>0$, then $b_2v_1+b_3v_2+b_4(-v_1-v_2)\geq 0$ and $b_1v_1+b_3v_2+b_4(-v_1-v_2)\geq 0$. Hence the matrices
\begin{equation*}
B_1:=\left(\begin{array}{cc}
b_1-b_3&b_2-b_3\\
b_2-b_4&b_3-b_4
\end{array}\right),~~
B_2:=\left(\begin{array}{cc}
b_1-b_3&b_2-b_3\\
b_1-b_4&b_3-b_4
\end{array}\right)
\end{equation*}
must have rank $\leq 1$. Assume $b_2-b_3\neq 0$. Then there are $x,y\in\RR$ satisfying $(b_2-b_4,b_3-b_4)=x(b_1-b_3,b_2-b_3)$, $(b_1-b_4,b_3-b_4)=y(b_1-b_3,b_2-b_3)$. As the second columns of $B_1$ and $B_2$ agree, we have $x=y$ and, thus, $b_1=b_2$. But then, $\rank(B_1)\leq 1$ implies $b_2=b_3$, which is a contradiction. We therefore may assume $b_2-b_3=0$. If $b_1-b_3\neq 0$, $\rank(B_1)\leq 1$ then yields $b_3-b_4=0$, so $b_2=b_3=b_4$. On the other hand, if $b_1-b_3=0$, we obtain $b_1=b_2=b_3$. As we assume $r_2$ to be orthogonal to $r_1$, this proves (\ref{second row}).
\end{proof}

\begin{lemma}
Let $i\in\NN$, $n\geq i$ and $\prec$ be a monomial preorder on $R_n$.
\begin{compactenum}[(1)]
\item If $x_1$ and $x_j$ are incomparable for all $j\in[i]$, then any two monomials $f,g\in R_i$ with $\deg(f)=\deg(g)$ are incomparable.
\item If $1$ and $x_j$ are incomparable for all $j\in[i]$, then every pair of monomials $f,g\in R_i$ is incomparable.
\end{compactenum}
\label{lemma incomparability}
\end{lemma}
\begin{proof}
Note that if $f,g$ and $f',g'$ are two pairs of incomparable monomials in $R_n$, then the pair $ff',gg'$ is incomparable, too. In case (1), this implies that $x_1^{\deg(f)}$ and $f$ are incomparable for every monomial $f\in R_i$; in case (2), we obtain that $1$ and $f$ are incomparable for all monomials $f\in R_i$. Transitivity of incomparability now yields the desired statements.
\end{proof}

For a monomial preorder $\prec$ on any polynomial ring, we denote by $\prec^{-1}$ the inverse of $\prec$, i.e. the monomial preorder which is defined by $f\prec^{-1}g$ $\Leftrightarrow$ $f\succ g$. We call a monomial preorder trivial if every pair of monomials $f,g$ is incomparable. A degree order is a monomial preorder $\prec$ satisfying $\deg(f)<\deg(g)\Rightarrow f\prec g$, and a reverse degree order is a monomial preorder which is the inverse of a degree order.

In what follows, we will write $f\preceq g$ instead of $f\not\succ g$, and we set $R_0:=\KK$.

\begin{proposition}
For monomials $x^a\neq x^b\in R_4$ let $A:=\lbrace j\,|\,a_j\neq b_j\rbrace$. The $\IInc(\NN)$-compatible, non-trivial monomial preorders on $R_4$ are:
\begin{compactenum}[(1)]
\item $\Deg$: $x^a\Deg x^b$ $:\Leftrightarrow$ $\deg(x^a)<\deg(x^b)$;
\item $\Min$: $x^a\Min x^b$ $:\Leftrightarrow$ $a_{\min(A)}<b_{\min(A)}$;
\item $\DegMin$: $x^a\DegMin x^b$ $:\Leftrightarrow$ $\deg(x^a)<\deg(x^b)$ or $(\deg(x^a)=\deg(x^b)$ and $a_{\min(A)}<b_{\min(A)})$;
\item $\RevdegMin$: $x^a\RevdegMin x^b$ $:\Leftrightarrow$ $\deg(x^a)>\deg(x^b)$ or $(\deg(x^a)=\deg(x^b)$ and $a_{\min(A)}<b_{\min(A)})$;
\item $\Maxi$, $i\in[4]$: $x^a\Maxi x^b$ $:\Leftrightarrow$ $\max(A)\geq i$ and $a_{\max(A)}<b_{\max(A)}$;
\item $\DegMaxi$, $i\in\lbrace 2,3,4\rbrace$: $x^a\DegMaxi x^b$ $:\Leftrightarrow$ $\deg(x^a)<\deg(x^b)$ or $(\deg(x^a)=\deg(x^b)$, $\max(A)\geq i$ and $a_{\max(A)}<b_{\max(A)})$;
\item $\RevdegMaxi$, $i\in\lbrace 2,3,4\rbrace$: $x^a\RevdegMaxi x^b$ $:\Leftrightarrow$ $\deg(x^a)>\deg(x^b)$ or $(\deg(x^a)=\deg(x^b)$, $\max(A)\geq i$ and $a_{\max(A)}<b_{\max(A)})$;
\end{compactenum}
and their inverses.
\label{proposition: preorders on R_4}
\end{proposition}
\begin{proof}
Let $\prec$ be an $\IInc(\NN)$-compatible monomial preorder on $R_4$ and let $r_j$ be the $j$th row of a matrix representing it. We first consider the case $r_1=(\lambda,\lambda,\lambda,\lambda)$, $r_2=(3\mu,-\mu,-\mu,-\mu)$ and the case $r_1=(\lambda,0,0,0)$, assuming that $\lambda,\mu>0$. In the first case, $\prec$ is a degree order with the additional property that ($a_1<b_1$ $\Rightarrow$ $x^a\prec x^b$) for monomials $x^a,x^b$ of the same degree. In the second case, this implication is valid for any pair of monomials $x^a,x^b$. Let $x^a,x^b\in R_4$ be monomials such that $a_{\min(A)}<b_{\min(A)}$ and, in the first case, $\deg(x^a)=\deg(x^b)$. We may assume that $a_i=b_i=0$ for $1\leq i<\min(A)$. Choose $p\in\IInc(\NN)$ with $\lbrace\min(A),...,4\rbrace\subseteq p([4])$ and $p(1)=\min(A)$. Then we have $p^{-1}\cdot x^a\prec p^{-1}\cdot x^b$, and by $\IInc(\NN)$-compatibility we conclude that this relation holds for $x^a$ and $x^b$, too. Thus, we obtain $\prec=\DegMin$ in the first and $\prec=\Min$ in the second case. If $\lambda<0$ or $\mu<0$, an analogous argument shows that in the first case, $\prec$ is one of the monomial preorders $(\RevdegMin)^{-1},\RevdegMin,(\DegMin)^{-1}$, and in the second case, we have $\prec=(\Min)^{-1}$.

Now assume that $r_1=(\lambda,\lambda,\lambda,\lambda)$ and $r_2=(-\mu,-\mu,-\mu,3\mu)$ with $\lambda,\mu>0$ (as above, the cases $\lambda<0$ or $\mu<0$ can be dealt with similarly). Then again, $\prec$ is a degree order, and for monomials $x^a,x^b$ with $\deg(x^a)=\deg(x^b)$ we have ($a_4<b_4$ $\Rightarrow$ $x^a\prec x^b$). By $\IInc(\NN)$-compatibility, the relation $x_1\preceq x_i$ holds for all $i\in[4]$. Let $i\in\lbrace 2,3,4\rbrace$ be minimal such that $x_1\prec x_i$, and let $f\in R_{i-1}$, $g\in R_i\setminus R_{i-1}$ be any monomials with $\deg(f)=\deg(g)$. Writing $g=g_1x_i^{e}$ with $g_1\in R_{i-1}$ and $f=f_1f_2$ such that $\deg(f_1)=\deg(g_1)$, Lemma \ref{lemma incomparability}(1) tells us that $f_1$ and $g_1$ are incomparable and $x_i^e\succ f_2$, hence we obtain $f\prec g$. Now let $i<j\leq 4$ and $f\in R_{j-1}$, $g\in R_j\setminus R_{j-1}$ be monomials of the same degree. Suppose that $f$ and $g$ are incomparable. As by $\IInc(\NN)$-compatibility we have $x_1\prec x_{j-1}$, this yields $x_1g\prec x_{j-1}f$. Let $p\in\IInc(\NN)$ be any function satisfying $p(j)=4$. Then by $\IInc(\NN)$-compatibility, we have $p\cdot x_1g\prec p\cdot x_{j-1}f$, which is a contradiction. Thus, Lemma \ref{lemma incomparability}(1) and the $\IInc(\NN)$-compatibility of $\prec$ yield $\prec=\DegMaxi$.

Finally, let $r_1=(0,0,0,\lambda)$ and assume that $\lambda>0$. Then $\prec$ satisfies ($a_4<b_4$ $\Rightarrow$ $x^a\prec x^b$), so in particular $1\prec x_4$ and, thus, by $\IInc(\NN)$-compatibility $1\preceq x_i$ for all $i\in[4]$. Let $i\in[4]$ be minimal such that $1\prec x_i$ and let $f\in R_{i-1}$, $g=g_1x_i^{e}$ with $g_1\in R_{i-1}$ and $e>0$ be any monomials. By Lemma \ref{lemma incomparability}(2), $f$ and $g_1$ are incomparable, so we obtain $f\prec g$. Now let $i<j\leq 4$ and $f\in R_{j-1}$, $g\in R_j\setminus R_{j-1}$. Supposing that $f$ and $g$ are incomparable, we obtain $g\prec x_{j-1}f$. Arguing as in the paragraph above, this contradicts the $\IInc(\NN)$-compatibility of $\prec$. Hence, Lemma \ref{lemma incomparability}(2) and the $\IInc(\NN)$-compatibility of $\prec$ let us conclude that $\prec=\Maxi$.
\end{proof}

\begin{bemerkung}
For $n=2$ and $n=3$ there is an infinite number of $\IInc(\NN)$-compatible monomial preorders on $R_n$: For $n=2$, choose any irrational number $\lambda>0$. Then the matrix
\begin{equation*}
A(\lambda):=\left(\begin{array}{cc}
1&\lambda
\end{array}\right)
\end{equation*}
defines an $\IInc(\NN)$-compatible term order on $R_2$, and if $0<\lambda'\neq \lambda$ is another irrational number, the term orders represented by $A(\lambda)$ and $A(\lambda')$ are distinct.\newline
For $n=3$, let $\lambda>1$ be any irrational number and consider the matrix
\begin{equation*}
B(\lambda):=\left(\begin{array}{ccc}
1&1&1\\
1+\lambda&-1&-\lambda
\end{array}\right).
\end{equation*}
Then one can easily check that $B(\lambda)$ represents an $\IInc(\NN)$-compatible term order on $R_3$, and for distinct irrational numbers $\lambda,\lambda'>1$, the term orders represented by $B(\lambda)$ and $B(\lambda')$ are distinct, too.
\label{remark: why R_4?}
\end{bemerkung}

\begin{satz}
The $\IInc(\NN)$-compatible monomial preorders on $R$ are the same as on $R_4$, with the exception that the number $i$ used in the definitions of preorders (5), (6) and (7) can take arbitrary values in $\NN$. In particular, there are only six $\IInc(\NN)$-compatible term orders on $R$, namely $\Min$, $\DegMin$, $(\RevdegMin)^{-1}$, $\MaxEins$, $\DegMaxZwei$, $(\RevdegMaxZwei)^{-1}$.
\label{theorem: preorders on R}
\end{satz}
\begin{proof}
Let $\prec$ be an $\IInc(\NN)$-compatible monomial preorder on $R$. Note that by $\IInc(\NN)$-compatibility, we either have $x_1\preceq x_i$ for all $i\in\NN$ or $x_1\succeq x_i$ for all $i\in\NN$. We will only consider the former case. By Lemma \ref{lemma incomparability}(1), if $x_1$ and $x_i$ are incomparable for all $i$, then so are $x_i$ and $x_j$ for every pair of natural numbers $i,j$. On the other hand, if $i\in\NN$ is minimal such that $x_1\prec x_i$, then Lemma \ref{lemma incomparability}(1) yields $x_{i-1}\prec x_i$, so by $\IInc(\NN)$-compatibility we obtain $x_k\prec x_l$ for all $l\geq i$ and $k<l$. Thus, in any case we have $x_i\preceq x_j$ for all $i\leq j$.

We will use the following notation: For a monomial $f\in R\setminus\lbrace 1\rbrace$, let $m(f)$ and $M(f)$ denote the minimal or, respectively, maximal index of a variable occurring in $f$. By $e(f)$ and $E(f)$ we denote the exponents of these variables in $f$. For $f=1$, we set $m(f)=\infty$, $M(f)=0$ and define $x_{0}:=x_{\infty}:=1$. By the above observation, we have $x_{m(f)}^{\deg(f)}\preceq f$, $x_{M(f)}^{\deg(f)}\succeq f$, $x_{m(f)}^{\deg(f)-E(f)}x_{M(f)}^{E(f)}\preceq f$ and $x_{m(f)}^{e(f)}x_{M(f)}^{\deg(f)-e(f)}\succeq f$. Therefore, in order to show that for any monomials $f,g\in R$ the relation $f\prec g$ holds, it suffices to show one of the following relations:
\begin{compactenum}[-]
\item $x_{M(f)}^{\deg(f)}\prec x_{m(g)}^{\deg(g)}$;
\item $x_{m(f)}^{e(f)}x_{M(f)}^{\deg(f)-e(f)}\prec x_{m(g)}^{\deg(g)}$;
\item $x_{M(f)}^{\deg(f)}\prec x_{m(g)}^{\deg(g)-E(g)}x_{M(g)}^{E(g)}$.
\end{compactenum}

In the remainder of the proof, we will regard each element $p\in\IInc(\NN)$ as a strictly increasing function $p:\NN\cup\lbrace 0,\infty\rbrace\rightarrow\NN\cup\lbrace 0,\infty\rbrace$ by setting $p(0):=0$ and $p(\infty):=\infty$. For any natural number $n$, let $[n]_0:=[n]\cup\lbrace 0\rbrace$ and $[n]_{\infty}:=[n]\cup\lbrace\infty\rbrace$.

Let $\prec'$ be the restriction of $\prec$ to $R_4$. We will first show that if $\prec'$ is a degree or a reverse degree order, then the same holds for $\prec$. Assume that $f'\prec' g'$ whenever $\deg(f')<\deg(g')$ and let $f,g\in R$ be any monomials satisfying $\deg(f)<\deg(g)$. Choose $p\in\IInc(\NN)$ with $M(f),m(g)\in p([2]_0)$. Then we have $p^{-1}\cdot x_{M(f)}^{\deg(f)}\prec p^{-1}\cdot x_{m(g)}^{\deg(g)}$, which by $\IInc(\NN)$-compatibility of $\prec$ implies that $x_{M(f)}^{\deg(f)}\prec x_{m(g)}^{\deg(g)}$ and, thus, $f\prec g$. If $f'\prec' g'$ whenever $\deg(f')>\deg(g')$, an analogous argument shows that $\prec$, too, satisfies this condition, and we are done.

Assume that $\prec'$ is not a total order. By Proposition \ref{proposition: preorders on R_4}, this implies that $\prec'$ is either a degree or a reverse degree order or $1$ and $x_1$ are incomparable.\newline
Suppose that $x_1$ and $x_i$ are incomparable for all $i\in\NN$. Then Lemma \ref{lemma incomparability} implies that either $\prec\in\lbrace\Deg,(\Deg)^{-1}\rbrace$ or, using the transitivity of incomparability, $\prec$ is trivial.\newline
Now let $i\in\NN$ be minimal such that $x_1\prec x_i$. If $i=2$, Proposition \ref{proposition: preorders on R_4} yields $\prec'=\MaxZwei$. Let $j\geq 5$ and $f\in R_{j-1}$, $g\in R_j\setminus R_{j-1}$ and choose $p\in\IInc(\NN)$ with $p(3)=M(g)$ and $M(f),m(g)\in p([3]_0)$. Then $p^{-1}\cdot x_{M(f)}^{\deg(f)}\prec p^{-1}\cdot x_{m(g)}^{\deg(g)-E(g)}x_{M(g)}^{E(g)}$, so, by $\IInc(\NN)$-compatibility, $x_{M(f)}^{\deg(f)}\prec x_{m(g)}^{\deg(g)-E(g)}x_{M(g)}^{E(g)}$. We conclude that $f\prec g$ and obtain $\prec=\MaxZwei$. Now assume $i>2$ and $\prec$ to be a degree or a reverse degree order. Choose any $j\geq i$ and monomials $f\in R_{j-1}$, $g\in R_j\setminus R_{j-1}$ with $\deg(f)=\deg(g)=:D$. Let $p\in\IInc(\NN)$ be any function satisfying $p(i)=M(g)$, $M(f),m(g)\in p([i]_0)$. Then by Lemma \ref{lemma incomparability}(1), $p^{-1}\cdot x_{M(f)}^{D-E(g)}$ and $p^{-1}\cdot x_{m(g)}^{D-E(g)}$ are incomparable and $p^{-1}\cdot x_{M(f)}^{E(g)}\prec p^{-1}\cdot x_{M(g)}^{E(g)}$, so $p^{-1}\cdot x_{M(f)}^{D}\prec p^{-1}\cdot x_{m(g)}^{D-E(g)}x_{M(g)}^{E(g)}$. By $\IInc(\NN)$-compatibility, we conclude $f\prec g$, which implies $\prec=\DegMaxi$ or $\prec=\RevdegMaxi$. If $i>2$ and $\prec$ is neither a degree nor a reverse degree order, $\prec'$ cannot be a degree or a reverse degree order, either, hence $1$ and $x_1$ are incomparable. Thus, choosing again any $j\geq i$ and arbitrary monomials $f\in R_{j-1}$, $g\in R_j\setminus R_{j-1}$ and letting $p\in\IInc(\NN)$ be as above, Lemma \ref{lemma incomparability}(2) yields incomparability of $p^{-1}\cdot x_{M(f)}^{\deg(f)}$ and $p^{-1}\cdot x_{m(g)}^{\deg(g)-E(g)}$ and we conclude $p^{-1}\cdot x_{M(f)}^{\deg(f)}\prec p^{-1}\cdot x_{m(g)}^{\deg(g)-E(g)}x_{M(g)}^{E(g)}$ and therefore $f\prec g$. This implies $\prec=\Maxi$.

Finally, suppose that $\prec'$ is a total order. If $\prec'\in\lbrace\MaxEins,\DegMaxZwei,\RevdegMaxZwei\rbrace$, let $f,g\in R$ be any monomials satisfying $M(g)>M(f)$ and, if $\prec'=\DegMaxZwei$ or $\prec'=\RevdegMaxZwei$, $\deg(f)=\deg(g)$. Let $p\in\IInc(\NN)$ be such that $M(f),m(g),M(g)\in p([3]_0)$. Then $p^{-1}\cdot x_{M(f)}^{\deg(f)}\prec p^{-1}\cdot x_{m(g)}^{\deg(g)-E(g)}x_{M(g)}^{E(g)}$, yielding $f\prec g$ and, thus, $\prec\in\lbrace\MaxEins,\DegMaxZwei,\RevdegMaxZwei\rbrace$. On the other hand, if $\prec'\in\lbrace(\Min)^{-1},(\DegMin)^{-1},(\RevdegMin)^{-1}\rbrace$, choose any monomials $f,g\in R$ with $m(f)<m(g)$ and, if $\prec'=(\DegMin)^{-1}$ or $\prec'=(\RevdegMin)^{-1}$, $\deg(f)=\deg(g)$. Let $p\in\IInc(\NN)$ be such that $m(f),M(f),m(g)\in p([3]_\infty)$. Then $p^{-1}\cdot x_{m(f)}^{e(f)}x_{M(f)}^{\deg(f)-e(f)}\prec p^{-1}\cdot x_{m(g)}^{\deg(g)}$, so $f\prec g$ and, therefore, $\prec\in\lbrace(\Min)^{-1},(\DegMin)^{-1},(\RevdegMin)^{-1}\rbrace$.
\end{proof}

\begin{bemerkung} Whereas by Corollary \ref{corollary: finitely many initial ideals}, the number of initial ideals of any $\IInc(\NN)$-invariant ideal in $R$ with respect to $\IInc(\NN)$-compatible term orders is finite, this is not true for the number of initial ideals with respect to $\IInc(\NN)$-compatible monomial preorders. Let $J$ be as in Remark \ref{example: infinitely many initial ideals with respect to arbitrary term orders} and choose any numbers $n<n'$. Then $x_1x_{n}^2=\iin_{\prec_{\mathrm{max,n}}}(x_1^2x_{n}+x_1x_{n}^2)\in\iin_{\prec_{\mathrm{max,n}}}(J)$. On the other hand, any polynomial $f\in J$ either contains both $x_1^2x_n$ and $x_1x_n^2$ or neither of the two monomials. As $x_1^2x_n$ and $x_1x_n^2$ are incomparable with respect to $\prec_{\mathrm{max,n'}}$, this implies that every element from $\iin_{\prec_{\mathrm{max,n'}}}(J)$ containing $x_1x_n^2$ must also contain $x_1^2x_n$ and, hence, $x_1x_n^2\not\in\iin_{\prec_{\mathrm{max,n'}}}(J)$. We conclude that the initial ideals $\iin_{\prec_{\mathrm{max,n}}}(J)$, $n\in\NN$, are pairwise distinct.
\label{example: infinite number of initial ideals with respect to Inc-preorders}
\end{bemerkung}

%--------------------------------------------------------------------------------------------------------------%
%--------------------------------------------------------------------------------------------------------------%

\end{document}